\documentclass[12pt]{amsart}
\usepackage{amssymb,amsthm,amsmath,amstext}
\usepackage{mathrsfs}  
\usepackage{bm}        
\usepackage{mathtools} 

\usepackage[left=1.28in,top=.8in,right=1.28in,bottom=.8in]{geometry}

\usepackage{graphicx}
\usepackage[all]{xy}

\theoremstyle{plain}
\newtheorem{theorem}{Theorem}[section]
\newtheorem{prop}[theorem]{Proposition}
\newtheorem{lemma}[theorem]{Lemma}

\newtheorem{cor}[theorem]{Corollary}

\theoremstyle{definition}
\newtheorem{defin}[theorem]{Definition}

\theoremstyle{remark}

\newcommand{\QQ}{\sheaf{Q}}

\newcommand{\Z}{\mathbb Z}

\newcommand{\Q}{\mathbb Q}

\usepackage{hyperref}
\usepackage{color}

  \usepackage[OT2, T1]{fontenc}
\DeclareSymbolFont{cyrletters}{OT2}{wncyr}{m}{n}
\DeclareMathSymbol{\Sha}{\mathalpha}{cyrletters}{"58}

\hypersetup{colorlinks=true,linkcolor=blue,anchorcolor=blue,citecolor=blue,letterpaper=true}

\begin{document}

\title[ strong $u$-invariant and period-index bound]
{Strong $u$-invariant and Period-Index bound for complete ultrametric fields}

\author[Shilpi Mandal]{Shilpi Mandal}
\address{Department of Mathematics\\ 
Emory University \\ 
400 Dowman Drive~NE \\ 
Atlanta, GA 30322, USA}
\email{shilpi.mandal@emory.edu}

\date{}

\begin{abstract}  Let $k$ be a complete ultrametric field with $\text{dim}_\Q {\sqrt{\mid k^* \mid}} = n$ finite, with residue field $\Tilde{k}$, and char$(\Tilde{k}) \neq 2$. We prove that $u(k) \leq 2^n u(\Tilde{k})$. Let $u_s(k)$ be the strong $u$-invariant of $k$, then we further show that $u_s(k) \leq 2^n u_s(\Tilde{k})$. Let $l$ be a prime such that the gcd(char$(\Tilde{k}), l$) $= 1$. If the $\text{Br}_l \text{dim} (\Tilde{k}) = d$, then we also show that $\text{Br}_l \text{dim} (k) \leq d + n$. Let $C$ be a curve over $k$ and $F = k(C)$. Then we show that for any quadratic form over $F$ with dimension $> 2^{n+1} u_s (\Tilde{k})$ is isotropic over $F$. We further show that if $\text{Br}_l \text{dim}(\Tilde{k}) \leq d$ and $\text{Br}_l \text{dim} (\Tilde{k}(T)) \leq d+1$, then $\text{Br}_l \text{dim}(F) \leq d+n+1$.
 
\end{abstract} 
 
\maketitle

\def\ZZ{${\mathbf Z}$}
\def\ih{${\mathbf H}$}
\def\RR{${\mathbf R}$}
\def\IF{${\mathbf F}$}
\def\QQ{${\mathbf Q}$}
\def\IP{${\mathbf P}$}

\date{}

\section{Introduction}

Let $k$ be a field with $\text{char}(k) \neq 2$. The \textit{$u$-invariant of $k$}, denoted by $u(k)$, is the maximal dimension of anisotropic quadratic forms over $k$. We say that $u(k) = \infty$ if there exist anisotropic quadratic forms over $k$ of arbitrarily large dimension. For example, $u$-invariant of an algebraic closed field $K$, $u(K) = 1$; $u(\mathbb{R}) = \infty$; for $k$ a finite field, $u(k) = 2$, etc. The $u$-invariant is a positive integer if it is finite. A key area of research in the theory of quadratic forms is to find all the possible values this invariant can take for a fixed or varying field. For example, it has been established in the literature (see \cite[Chapter XI, Proposition 6.8]{Lam}) that the $u$-invariant cannot take values $3$, $5$, and $7$. See \cite[Chapter XIII, Section 6]{Lam}, for more open problems about this invariant. Considerable progress has also been made, particularly in the computation of the $u$-invariant of function fields of $p$-adic curves due to Parimala and Suresh in \cite{PS10}, \cite{PS14}, and by Harbater, Hartmann, and Krashen regarding the $u$-invariants in the case of function fields of curves over complete discretely valued fields in \cite{HHK09}.

Harbater, Hartmann, and Krashen defined the \textit{strong $u$-invariant of $K$}, denoted by $u_s(K)$, as the smallest real number $m$ such that, $u(E) \leq m$ for all finite field extensions $E/K$, and $u(E) \leq 2m$ for all finitely generated field extensions $E/K$ of transcendence degree $1$. We say that $u_s(K) = \infty$ if there exist such field extensions $E$ of arbitrarily large $u$-invariant. In \cite[Theorem 4.10]{HHK09}, the same authors prove that for $K$ a complete discretely valued field, whose residue field $\Tilde{K}$ has characteristic unequal to 2, $u_s(K) = 2 u_s(\Tilde{K})$.

Let $k$ be a complete non-Archimedean valued field with residue field $\Tilde{k}$ such that $\text{char}(\Tilde{k}) \neq 2$. 
Let $\sqrt{|k^*|}$ denote the divisible closure of the value group $|k^*|$. In \cite{Meh19}, Mehmeti shows that if $\text{dim}_\Q \sqrt{|k^*|} = n$, then $u_s(k) \leq 2^{n+1} u_s(\Tilde{k})$, and if $|k^*|$ is a free $\Z$-module with $\text{rank}_\Z |k^*| = n$, then $u_s(k) \leq 2^n u_s(\Tilde{k})$. Mehmeti used field patching in the setting of Berkovich analytic geometry, to prove a local-global principle, and provides applications to quadratic forms and the $u$-invariant. The results she obtained generalises those of \cite{HHK09}.
 
\bigskip

Our main result concerning $u$-invariant of a complete non-Archimedean (ultrametric) field is the following theorem.
\begin{theorem}[Theorem \ref{leq 2^n}]
    Let $k$ be a complete ultrametric field with $\text{char}(\Tilde{k}) \neq 2$. Suppose that $\text{dim}_\Q (\sqrt{\mid k^* \mid}) = n$ is finite. Then $u(k) \leq 2^n u(\Tilde{k})$.
\end{theorem}

\begin{cor}[Theorem \ref{leq 2^n+1}]
    Let $k$ be a complete ultrametric field with $\text{char}(\Tilde{k}) \neq 2$. Suppose the $\text{dim}_\Q (\sqrt{|k^*|}) = n$ is finite. Let $C$ be a curve over $k$ and $F = k(C)$ the function field of $C$. Then $u(F) \leq 2^{n+1} u_s(\Tilde{k})$. 
\end{cor}


Given a field $k$, recall the definitions of period and index of a central simple algebra. The \textit{period} (or \textit{exponent}) of a central simple $k$-algebra $A$ is the order of class of $A$ in the Brauer group of $k$. The \textit{index} of $A$ is the degree of the division algebra $D_A$ that lies in the class of $A$ (i.e., $A$ is a matrix ring over $D_A$). The period and index always have the same prime factors, and the period always divides the index \cite[Proposition 4.5.13]{GS}. 

The \textit{period-index problem} asks whether all central simple algebras $A$ over a given field $k$ satisfy ind($A$) | per($A$)$^d$ for some fixed exponent $d$ depending only on $k$. In the spirit of \cite{PS14}, we make the following definition. Let $k$ be any field. For a prime $l$, define the \textit{Brauer $l$-dimension} of $k$, denoted by Br$_l$dim($k$), to be the smallest integer $d \geq 0$ such that for every finite extension $L$ of $k$ and for every central simple algebra $A$ over $L$ of period a power of $l$, ind($A$) divides per($A$)$^d$. The \textit{Brauer dimension of $k$}, denoted by $\text{Brdim}(k)$, is defined as the maximum of the Brauer $l$-dimension of $k$ as $l$ varies over all primes. It is expected that this invariant should increase by one upon a finitely generated field extension of transcendence degree one. 

\bigskip

Saltman proved some results in this direction, including the fact that the index divides the period squared for function fields of $p$-adic curves \cite{Salt97}. He also developed a general mechanism to relate the Brauer dimension of curves over discretely valued fields to that of curves over the residue field. Then Harbater, Hartmann, and Krashen in \cite[Theorem 5.5]{HHK09} show that for $k$ a complete discretely valued field, its residue field $\Tilde{k}$, $F$ the function field of a curve over $k$, and $l \neq \text{char}(\Tilde{k})$, they prove that if $\text{Br}_l \text{dim} (\Tilde{k}) \leq d$ and $\text{Br}_l \text{dim}(\Tilde{k}(T)) \leq d+1$, then $\text{Br}_l \text{dim}(F) \leq d+2$. 

\bigskip

Using \cite{Meh19} we get the following as a direct consequence.
\begin{theorem}[Theorem \ref{local-global index}]
    Let $k$ be a complete ultrametric field with $\text{dim}_\mathbb{Q} (\sqrt{|k^*|}) = n$ is finite. Let $C$ be a curve over $k$ and $F = k(C)$ the function field of the curve. Let $A$ be a central simple algebra over $F$ and let $V(F)$ be the set of all non-trivial rank 1 valuations of $F$. Then $\text{ind}(A)$ is the maximum of the set $\{ \text{ind}(A \otimes \widetilde{F_v}) \}$ for $v \in V(F)$.
\end{theorem}

Our main result concerning the Brauer dimension are the following theorems. 

\begin{theorem}[Theorem \ref{leq d+n}]
    Let $k$ be a complete ultrametric field. Let $l$ be a prime such that $l \neq \text{char}(\Tilde{k})$. Suppose $\text{dim}_\mathbb{Q} (\sqrt{|k^*|}) = n$ is finite. If the $\text{Br}_l {dim} (\Tilde{k}) = d$, then $\text{Br}_l {dim}(k) \leq d + n$.
\end{theorem}

\begin{theorem}[Theorem \ref{leq d+n+1}]
    Let $k$ be a complete ultrametric field. Let $l$ be a prime such that $l \neq \text{char}(\Tilde{k})$. Suppose $\text{dim}_\mathbb{Q} (\sqrt{|k^*|}) = n$ is finite. Let $C$ be a curve over $k$ and $F = k(C)$ the function field of $C$. Suppose there exist an integer $d$ such that $\text{Br}_l \text{dim} (\Tilde{k}) \leq d$ and $\text{Br}_l \text{dim} (\Tilde{k}(T)) \leq d+1$. Then the $\text{Br}_l \text{dim} (F) \leq d+1+n$.
\end{theorem}

\textit{Remark:} Note that if $k$ is a complete discretely valued field, then $n = 1$. Theorem 1.4 now gives that $\text{Br}_l \text{dim} (k) \leq d+1$, which is a classical result \cite[Corollary 7.1.10]{GS}. In the same case, Theorem 1.5 now implies that if $l \neq \text{char}(\Tilde{k})$, then $\text{Br}_l \text{dim}(F) \leq d+2$, which is found in \cite[Corollary 5.10]{HHK09}.

\subsection{Organisation of the manuscript}
We start with some preliminaries and necessary background in Section 2, where we talk about the known results on local-global principles from \cite{HHK09} and \cite{Meh19}.

In Section 3 we prove the main decomposition lemma, which we then use to prove that for a complete ultrametric field $k$ with residue field characteristic $\text{char}(\Tilde{k}) \neq 2$ and $\text{dim}_\Q (\sqrt{|k^*|}) = n$, $u(k) \leq 2^n u(\Tilde{k})$. We further prove the result about $\text{Br}_l\text{dim}(k) \leq \text{Br}_l\text{dim}(\Tilde{k})+n$.

In Section 4, using the results from Section 3 and the results of Mehmeti we prove our results on $u$-invariant and Brauer $l$-dimension of function fields of curves over complete ultrametric fields.

 \section{Preliminaries}

Let $F$ be a field. For a central simple $F$-algebra $A$, the \textit{period} of $A$ is the order of its class in the Brauer group of $F$, and the \textit{index} of $A$ is the degree of the division algebra Brauer equivalent to $A$. The index of $A$ is denoted by ind($A$) and period of $A$ by per($A$). Also, per($A$) divides ind($A$), and the two numbers are composed of the same prime factors.

If a linear algebraic group $G$ acts on a variety $H$ over a field $F$, we will say that $G$ \textit{acts transitively on the points of} $H$ if for every field extension $E$ of $F$ the induced action of the group $G(E)$ on the set $H(E)$ is transitive. 
\bigskip

The following definition is due to F. Ch$\hat{\text{a}}$telet. For a more detailed discussion of results on Severi-Brauer varieties see Section 7.1 in \cite{CTS}.

\begin{defin} [Severi-Brauer variety]
    Let $n$ be a positive integer. A Severi-Brauer variety of dimension $n-1$ over a field $k$ is a twisted form of the projective space $\mathbb{P}^{n-1}_k$. Equivalently, this is a $k$-variety $X$ such that there exists a field extension $k \subset K$ and an isomorphism of $K$-varieties $X \times_k K \cong \mathbb{P}^{n-1}_k$.
\end{defin}
 There is a natural bijection between the isomorphism classes of Severi-Brauer varieties over a field $k$ and the isomorphism classes of central simple $k$-algebras.
 
 For instance, a twisted form of $\mathbb{P}^1_k$ is a smooth, projective, geometrically integral curve $C$ of genus 0, and vice versa, any smooth plane conic is a twisted form of $\mathbb{P}^1_k$. The automorphism group of projective space $\mathbb{P}^{n-1}_k$ is the algebraic group PGL$_{n,k}$ \cite[Ch. II, Example 7.1.1]{Hartshorne}. The automorphism group of the matrix algebra $M_n(k)$ is the projective general linear group PGL$_n (k)$ \cite[Corollary 2.4.2]{GS}. Galois descent then gives a bijection between the isomorphism classes of twisted forms of $\mathbb{P}^{n-1}_k$ and the isomorphism classes of twisted forms of $M_n(k)$ (see the discussion in \cite[Section 1.3.2]{CTS}), plus we know that the matrix algebra $M_n(k)$ is always a central simple algebra over $k$. Thus we get a canonical bijection of pointed sets \cite[Theorem 2.4.3]{GS}
 \[ H^1 (k, \text{PGL}_{n,k} ) \cong \{ \text{central simple algebras over $k$ of degree $n$}\}/\text{iso}\\ \]\[ \cong \{ \text{Severi-Brauer varieties over $k$ of dimension $n-1$} \}/ \text{iso} \] and a map of pointed sets $H^1 (k, \text{PGL}_{n,k} ) \to \text{Br}(k)$, where a central simple algebra $A$ of degree $n$ is sent to its class $[A] \in \text{Br}(k)$. In the case of a Severi-Brauer variety $X$ of dimension $n-1$, denote by $[X] \in \text{Br}(k)$ the image of the isomorphism class of $X$ under this map.

 For a central simple algebra $A$ over $k$ of degree $n$, define $X(A)$ to be the $k$-scheme of right ideals of $A$ of rank $n$. More precisely, for any commutative $k$-algebra $R$, the set $X(A)(R)$ is the set of right ideals of the $k$-algebra $A \otimes_k R$, which are projective $R$-modules of rank $n$ and are direct summands of the $R$-module $A \otimes_k R$. This is a closed subscheme of the Grassmannian of $n$-dimensional subspaces of the $k$-vector space $A$. For more details, see the discussion in \cite[Chapter 1, \S1.C]{book of involutions}.


Let $A$ be a central simple $F$- algebra of degree $n$ and let $I \subseteq A$ be a right ideal of $A$. Then the $F$-dimension of $I$, denoted by $\text{dim}_F (I)$ is divisible by the degree of $A$, deg($A$). The quotient $\text{dim}_F (I)/\text{deg}(A)$ is called the \textit{reduced dimension} of the ideal (see Definition 1.9 in \cite{book of involutions}). For any integer $i$, we write SB$_i (A)$ for the $i$-\textit{th generalized Severi-Brauer variety} of the right ideals in $A$ of reduced dimension $i$. In particular, SB$_0 (A)$ $=$ Spec($F$) $=$ SB$_{\text{deg}(A)} (A)$ and SB$_i (A) = \phi$ for $i$ outside of the interval $[0, \text{deg}(A)]$. The variety SB$_1 (A)$ is the usual Severi-Brauer variety of A.

In the same setup let $E/F$ be a field extension and if $A_E$ denotes the algebra $A \otimes_F E$, then the $E$-points of $\text{SB}_i(A)$ are in bijection with the right ideals in $A_E$ which are of reduced dimension $i$ and are direct summands of $A_E$.
Also, $\text{SB}_i(A)(E) \neq \phi$ if and only if ind($A_E$) divides $i$ (see Chapter 1, \S1.C in \cite{book of involutions}). Since $A_E$ is a central simple algebra over $E$, we have $A_E \cong \text{M}_m (D)$ for some $E$-division algebra $D$ and some $m \geq 1$. 
Now, the right ideals of reduced dimension $i$ in $A_E$ are in natural bijection with the subspaces of $D^m$ of $D$-dimension $i/\text{ind}(A_E)$ \cite[Proposition 1.12]{book of involutions}. Thus, writing $D_A$ for the $F$-division algebra in the class of $A$, the $F$-linear algebraic group GL$_1$($A$) $=$ GL$_m$($D_A$) acts transitively on the points of the $F$-scheme $\text{SB}_i (A)$.

\begin{defin} \cite{HHK09}
Let $K$ be a field. Let $X$ be a $K$-variety and $G$ a linear algebraic group over $K$. We say that $G$ acts on $X$ and, for any field extension $L/K$, either $X(L) = \phi$ or $G(L)$ acts transitively on $X(L)$.  
\end{defin}

A field $k$ equipped with a non-Archimedean valuation $|\cdot| : k \to \mathbb{R}_{\geq 0}$ is called a \textit{non-Archimedean field} or an \textit{ultrametric field}. We say $k$ is \textit{complete} if every Cauchy sequence converges to a limit.

In this paper, we use field patching in the setting of Berkovich analytic geometry. By patching over analytic curves, Mehmeti proved a local-global principle and provided applications to quadratic forms and the u-invariant.

We use the following main result of Mehmeti \cite{Meh19} on the local-global principle in Berkovich theory.

\begin{theorem}[Mehmeti \cite{Meh19}] \label{LGP Meh}
    Let $k$ be a complete ultrametric field. Let $C$ be a normal irreducible projective $k$-algebraic curve. Denote by $F$ the function field of $C$. Let $X$ be an $F$-variety and $G$ a connected rational linear algebraic group over $F$ acting strongly transitively on $X$.

    Let $V(F)$ be the set of all non-trivial rank 1 valuations of $F$ which either extend the valuation of $k$ or are trivial when restricted to $k$.

   If $F$ is a perfect field or $X$ is a smooth variety, then \[X(F) \neq \phi\ \text{if and only if}\ X(F_v) \neq \phi\ \text{for all v} \in V(F),\] where $F_v$ denotes the completion of $F$ with respect to $v$.
\end{theorem}
\bigskip

The following theorem follows from the proof of Harbater, Hartmann, and Krashen's Theorem 5.1 in \cite{HHK09} and Mehmeti's Corollary 3.19 in \cite{Meh19}. 

\begin{theorem} \label{local-global index}
    Let $k$ be a complete, non-trivially valued ultrametric field and suppose $\text{dim}_\mathbb{Q} (\sqrt{|k^*|}) = n$ is finite. Let $C$ be a normal irreducible projective $k$-algebraic curve. Denote by $F$ the function field of $C$. 
    Let $A$ be a central simple algebra over $F$ and let $V(F)$ be the set of all non-trivial rank 1 valuations of $F$. Then $\text{ind}(A)$ is the maximum of the set $\{ \text{ind}(A \otimes F_v) \}$ for $v \in V(F)$.
\end{theorem}

\begin{proof}
    Let $n$ be the degree of $A$. Then GL$_1 (A)$ is a Zariski open subset of $\mathbb{A}_F^{n^2}$. So, it is a rational connected linear algebraic group.

    If $1 \leq i < n$, then GL$_1 (A)$ acts transitively on the points of SB$_i (A)$; and if $E$ is a field extension of $F$, then SB$_i (A) (E) \neq \phi$ if and only if ind($A_E$) divides $i$ \cite[Proposition 1.17]{book of involutions}. So, the local-global principle proved by Mehmeti in \cite{Meh19} implies that ind($A$) | $i$ if and only if ind($A \otimes F_v) | i$ for each $v \in V(F)$. Thus ind($A$) $=$ max$_{v \in V(F)}$ \{ind($A \otimes F_v$)\} as claimed.
\end{proof}

We use the following notation throughout the paper.

\textit{Notation}: For any valued field $E$, let $E^\circ$ denote the ring of integers, $E^{\circ \circ}$ the corresponding maximal ideal, and $\widetilde{E}$ the residue field. Let $l$ be a prime such that $gcd ( l, char(\Tilde{k}) ) = 1$.
\bigskip
   
We will need the following theorem (stated here without proof) to develop the rest of the paper. As in Milnor's Algebraic $K$-Theory \cite{Mil71}, let $L$ be an arbitrary field, $n$ a positive integer with a non-zero image in $L$. The $K$-theory of $L$ forms a graded anti-commutative ring $\displaystyle\oplus_{i \geq 0} K_i(L)$ with $K_0(L) \cong \Z$, $K_1 (L) \cong \text{GL}_1 (L)$. Denote by $(a)$ for an element of $K_1(L)$ that corresponds with $a \in L^*$, such that $(a) + (b) = (ab)$. Multiplying $(a), (b) \in K_1(L)$ yields an element of $K_2(L)$ that is written as $(a,b)$. One calls $(a,b)$ a Steinberg symbol, or simply, a \textit{symbol}.

In literature, there are two descriptions of the main theorem of Merkurjev and Suslin, one with Galois cohomology and one with Brauer groups. Let $\mu_n$ be the $n$-th roots of unity in an algebraic closure $\Bar{L}$ of $L$ and denote by $L^s$ the separable closure of $L$ in $\Bar{L}$. Now consider $\mu_n$ as a module for the absolute Galois group Gal$(L^s/L)$. By Hilbert Theorem 90, the Galois cohomology group $H^1 (L, \mu_n)$ is isomorphic with $k_2(L) := K_1(L)/n K_1(L)$ (see \cite[Section 8.4]{GS}). Using a theorem of Matsumoto, 
and the product structures in $K$-theory and in Galois cohomology, it gives the following homomorphism called the \textit{norm-residue homomorphism} or the \textit{Galois symbol}
\[
\alpha_{L,n}: k_2 (L) \to H^2(L, \mu_n^{\otimes 2}).
\] The Merkurjev-Suslin theorem now simply reads

\begin{theorem}[Merkurjev-Suslin \cite{MS82}]
    For all $L, n$ as above, $\alpha_{L,n}$ is an isomorphism.
\end{theorem}

Recall the definition of a cyclic algebra over a field. 
For a positive integer $n$, let $L$ be a field in which $n$ is invertible such that $L$ contains a primitive $n$-th root of unity $\omega$. For nonzero elements $a$ and $b$ of $L$, the associated \textit{cyclic algebra} is the central simple algebra of degree $n$ over $L$ defined by \[(a,b)_{\omega} = L\langle x,y\rangle/(x^n=a, y^n=b, xy=\omega yx).\]

We know from literature that the $n$-torsion of the Brauer group, $_n\text{Br}(L)$ is isomorphic with $H^2(L, \mu_n^{\otimes2})$ \cite[Section 1.3.4]{CTS}. So, we find that $\alpha_{L,n} : k_2(L) \to \ _n\text{Br}(L)$ sends the coset of the symbol $(a,b)$ to the class of the the cyclic algebra $(a,b)_\omega$, where $\omega$ is the primitive $n$-th root of unity in $L$. We thus get the following corollary, which conveys the surjectivity of $\alpha_{L,n}$.
\begin{cor}
    Let $\mu_n \subset L$ and let $A$ be a central simple algebra over $L$, with $[A] \in\ _n\text{Br}(L)$. Then $A$ is similar to a tensor product of cyclic algebras $(a_i, b_i)_\omega$, $a_i, b_i \in L^*$.
\end{cor}
The case $n=2$ is actually Merkurjev's theorem \cite{Mer81}.
\bigskip

To define a cyclic algebra over a ring $R$, let $n$ be an invertible positive integer in $R$. Suppose $R$ contains a primitive $n$-th root of unity $\omega$. For $a, b \in R^*$, the associated cyclic algebra is an Azumaya algebra of degree $n$ over $R$ defined by \[(a,b)_{\omega} = R\langle x,y\rangle/(x^n=a, y^n=b, xy=\omega yx).\]

The notion of central simple algebra over a field generalises to that of an \textit{Azumaya algebra} over a commutative ring. The following theorem is due to Azumaya (over a local ring), Auslander and Goldman (over an arbitrary commutative ring), and Grothendieck (over a scheme).

\begin{defin}[Azumaya algebra] \label{azumaya def}
    Let $R$ be a commutative ring and $A$ a ring over $R$ that is finitely generated and projective (equivalently, locally free) as an $R$-module. Let $A_S:= S \otimes_R A$ for a homomorphism of rings $R \to S$. A ring $A$ satisfying any of the following equivalent conditions is called an \textit{Azumaya algebra} over $R$.
    \begin{itemize}
        \item[(i)] The map $A \otimes_R A^{\text{op}} \to \text{End}_R (A)$ is an isomorphism.
        \item[(ii)] For every algebraically closed field $k$ and a homomorphism $R \to k$, $A_k \simeq \text{Mat}_n(k)$.
        \item[(iii)] For every maximal ideal $\mathfrak{m} \subset R$, let $k = R/\mathfrak{m}$; the ring $A_k$ is a central simple algebra over $k$.
    \end{itemize}

\end{defin}

An Azumaya algebra over a commutative ring $R$ is an algebra over $R$ that has an inverse up to Morita equivalence. That is, $A$ is an Azumaya algebra if there is an $R$-algebra $B$ such that $B \otimes_R A$ is Morita equivalent to $R$, which is the unit for the tensor product of $R$-algebras. Thus, Morita equivalence classes of Azumaya algebras over $R$ form a group, which is called the Brauer group of $R$, denoted by $\text{Br}(R)$.

\bigskip

We also record here the following well-known result. 
\begin{prop} \label{index of azumaya algebras} Let $R$ be a complete local domain with field of fractions $F$ and residue field $k$.
Let $A$ be an Azumaya algebra over $R$. Then 
$$ind(A \otimes_R F) \leq ind(A\otimes_R k).$$ 
\end{prop}
 
 \begin{proof} Note that $A \otimes_R k$ is a central simple algebra over $k$ \cite[Corollary 2.6]{Pres15}. Let $n  = $ ind$(A \otimes_R k)$. Then there exists a separable field extension $L/k$ of degree $n$ such that $(A \otimes_R k) \otimes L \simeq M_n(L)$  \cite[Proposition 2.25]{GS}.
Since $L/k$ is a finite separable extension, $L = k(\alpha)$ for some $\alpha \in L$. 
Let $ f(x) \in k[X]$ be the monic minimal polynomial of $\alpha$ over $k$.
 Let $g(X) \in R[X]$ be a monic polynomial of degree $n$ which maps to $f(X)$ in $k[X]$. 
 Since $f(X)$ is irreducible in $k[X]$, $g(X)$ is irreducible in $R[X]$. Let $S = R[X]/(g(X))$ and  $E = F[X]/(g(X))$.
 Since $R$ is a complete local ring, $S$ is
 also a complete local ring with residue field $L$ (see the proof of \cite[Theorem 6.3]{AG60}). 
  Since the map $Br(S) \to Br(L)$ is  injective
 \cite[Corollary 6.2]{AG60} and $(A\otimes_R S) \otimes_S L \simeq (A\otimes_R k) \otimes_k L$ is a matrix algebra,
 $A\otimes_R S$  represents the trivial element in $Br(S)$. Hence $(A\otimes_RF) \otimes E$ is a matrix algebra. 
 In particular ind$(A\otimes_R F) \leq [ E : F] = n$. 
 \end{proof}

The following discussion can be found in Commutative Algebra by Bourbaki \cite[Chap. 6, $\S$10, n◦2]{Bo}. For an abelian group $\Gamma$, define the \textit{rational rank} of $\Gamma$, denoted by \textit{rat.rank$(\Gamma)$}, to be the $\text{dim}_\Q (\Gamma \otimes_\Z \Q)$. The rational rank is an element of $\mathbb{N} \cup \{+\infty\}$. If $\Gamma'$ is a subgroup of $\Gamma$, then since $\Q$ is a flat $\Z$-module, we have that \textit{rat.rank$(\Gamma)$} $=$ \textit{rat.rank$(\Gamma')$} $+$ \textit{rat.rank$(\Gamma/\Gamma')$}. The rational rank of a group $\Gamma$ is zero if and only if $\Gamma$ is a torsion group. If $\Gamma$ is the value group of a valuation, then its rational rank is zero if and only if the valuation is the trivial valuation. Also, if the field $L$ is an algebraic extension of the valued field $k$, then the quotient group $\displaystyle |L^*|/|k^*|$ is a torsion group, and the residue field $\Tilde{L}$ is an algebraic extension of $\Tilde{k}$ \cite[Proposition 1.16]{MVaq}.

\section{Complete ultrametric fields}

Let $M \subset \mathbb{R}^*$ be a subgroup and let $\sqrt{M} := \{ a \in \mathbb{R}^* |\ a^z \in M\ \text{for some}\ z \in \Z_{\neq 0} \}$ be the divisible closure of $M$. It is a $\mathbb{Q}$-vector space, and suppose $dim_\mathbb{Q} (\sqrt{M})$ is finite. For a prime $l$, set $M^l = \{ m^l | m \in M \}$.

We want to understand what do elements of $M^* / {M^*}^l$ look like. 

There exists a $\mathbb{Q}$-basis $z_1, z_2, \dots, z_n$ of $\sqrt{M}$. By definition of $\sqrt{M}$, there exist non-zero integers $\alpha_1, \alpha_2, \dots, \alpha_n$ such that $z_1^{\alpha_1}, z_2^{\alpha_2}, \dots, z_n^{\alpha_n}$ are in $M$. Let $t_i := z_i^{\alpha_i} \in M$. 

Then $t_1, \dots, t_n$ is also a basis.
Thus, for any $t \in M$, there exist unique $p_1, p_2, \dots, p_n \in \mathbb{Q}$ such that 
\[
t = \prod_{i = 1}^n {t_i}^{p_i}.
\]


\begin{defin}
    Let $t \in M$. Let $t = \displaystyle\prod_{i=1}^n {t_i}^{s_i/r_i}$ for $\displaystyle\frac{s_i}{r_i} \in \mathbb{Q}$ with $gcd(s_i, r_i) = 1$ for $i = 1, 2, \dots, n$. Let $r$ be the lcm of $r_i, i = 1, 2, \dots, n$. We will say that $r$ is the \textit{order of $t$}.
\end{defin}

Let $k$ be an ultrametric field and $|\cdot|$ be the non-Archimedean valuation on $k$. Suppose that dim$_\Q \sqrt{| k^* |}$ (i.e., the rational rank of $|k^*|$) is finite, say $n=\text{dim}_\Q(\sqrt{|k^*|})$.

\bigskip
In particular, from the above discussion we get that when $M = |k^*|$, there exist $\pi_1, \pi_2, \dots, \pi_n \in k^*$ with $|\pi_i| = t_i$ such that for any $t \in \sqrt{|k^*|}$, there exist unique $p_1, p_2, \dots, p_n \in \Q$ such that
\[
t = \prod_{i=1}^n {|\pi_i|}^{p_i} = \prod_{i=1}^n t_i^{p_i}.
\]
Let us fix such elements $\pi_1, \pi_2, \dots, \pi_n$ of $k^*$.

\begin{lemma}
    Let $k$ be a field and $m,n$ be positive integers. Suppose $m$ is coprime to $n$. Let $a \in k^*$. Then there exists an integer $m'$ such that $a = a^{m m'} \in k^*/(k^*)^n$.
\end{lemma}
\begin{proof}
    Since $m$ and $n$ are coprime, there exists $m', n' \in \Z$ such that $m m' + n n' = 1$. Hence $a = a^{mm' + nn'} = a^{mm'} \in k^*/(k^*)^n$.
\end{proof}

\begin{theorem}[Main Decomposition Lemma]\label{main decomposition}
    Let $k$ be an ultramteric valued field, $rank_\mathbb{Q} (\sqrt{|k^*|}) = n$, and $l$ a prime such that $gcd (l, char(\Tilde{k})) = 1$. Let $a_1, a_2, \dots, a_s \in k^*$, there exist elements $c_1, c_2, \dots, c_n$ in $k^*$ such that any $a_i =\displaystyle u_i \prod_{j=1}^n c_j^{\mu_{ij}} (b_i)^l$, $\mu_{ij} \in \mathbb{Z}$, $0 \leq \mu_{ij} \leq l-1$ , and $b_i \in k$.
\end{theorem}

\begin{proof}
    Let us fix elements $\pi_1, \pi_2, \dots, \pi_n$ in $k^*$ such that $\{t_i = |\pi_i|\}$ is a basis of $\sqrt{|k^*|}$. 

    
    Since $a_i \in k^*$,$|a_i| \in \sqrt{|k^*|}$. So there exist unique $p_1, p_2, \dots, p_n \in \mathbb{Q}$ such that, $|a_i| = \displaystyle \prod_{j=1}^n |\pi_j|^{p_{ij}} = \prod_{j=1}^n t_j^{p_{ij}} =  t_1^{r_{i1}} \prod_{j=2}^n t_j^{r_{ij}}$, where $r_{ij} \in \Q$. Let $y_{i1}$ be the highest power of $l$ that divides the denominator of $r_{i1}$. Then we can write $|a_i| = t_1^{x_{i1}/l^{y_{i1}}z_{i1}} \prod_{j=2}^n t_j^{r_{ij}}$, where $x_{ij}, y_{i1}, z_{i1} \in \Z$, and $r_{ij} \in \Q$. 
    
    Since $l \nmid z_{i1}$, from the previous lemma, $a_i \equiv a_i^{z_{i1}}(\mod{(k^*)^l})$.

    Without loss of generality, we can say that any $a_i \in k^*$ has $\displaystyle |a_i^{z_{i1}}| = t_1^{x_{i1}/l^{y_{i1}}} \prod_{j=2}^n t_j^{r_{ij}}$, where $r_{ij} \in \Q$. Note that gcd($x_{i1}, l^{y_{i1}}$) $= 1$.

    Case(i): Suppose $y_{i1} = 0$ for all $i$. So, $\displaystyle |a_i| = t_1^{x_{i1}} \prod_{j=2}^n t_j^{r_{ij}}$. Considering the element $\pi_1^{-x_{i1}} a_i \in k^*$, $\displaystyle |\pi_1^{-x_{i1}}a_i| = \prod_{j=2}^n t_j^{r_{ij}}$, where $r_{ij} \in \Q$.

    Case(ii): Suppose $y_{i1} \geq 1$ for some $i$. Without loss of generality, say $y_{11} = \max_i \{y_{i1}\} \geq 1$. So, $\displaystyle |a_1| = t_1^{x_{11}/l^{y_{11}}} \prod_{j=2}^n t_j^{r_{1j}}$. We also know that $l \nmid x_{11}$, so we can choose $x_{11}', x_{11}'' \in \Z$ such that $l \nmid x_{11}'$ and $x_{11} x_{11}' + l^{y_{11}+1} x_{11}'' = 1$.

    So, $\displaystyle a_1 \equiv a_1^{x_{11}'} (\mod{(k^*)^l})$, implying $\displaystyle |a_1^{x_{11}'}| = t_1^{x_{11} x_{11}'/l^{y_{11}}} \prod_{j=2}^n t_j^{r_{1j} x_{11}'}$. Now consider the exponent of $t_1$:
    \[
    \frac{x_{11} x_{11}'}{l^{y_{11}}} = \frac{1 - l^{y_{11} + 1} x_{11}''}{l^{y_{11}}} = \frac{1}{l^{y_{11}}} - l x_{11}''.
    \]
    So $\displaystyle |a_1| \equiv |a_1^{x_{11}'}| \equiv t^{1/l^{y_{11}}} \prod_{j=2}^n t_j^{\lambda_{ij}} (\mod{(k^*)^l})$. Without loss of generality, say $\displaystyle |a_1| = t_1^{1/l^{y_{11}}} \prod_{j=2}^n t_j^{r_{1j}}$ and $\displaystyle |a_2| = t_1^{x_{21}/l^{y_{21}}} \prod_{j=2}^n t_j^{r_{2j}}$. Now consider the element $\displaystyle |a_1^{- x_{21} l^{y_{11} - y_{21}}} a_2| = \displaystyle t_1^{\frac{-l^{y_{11} - y_{21}}}{l^{y_{11}}} + \frac{x_{21}}{l^{y_{21}}}} \prod_{j=2}^n t_j^{\lambda_j}$, where $\lambda_j \in \Q$.

    Consider again the exponent of $t_1$:
    \[
    \frac{-x_{21} l^{y_{11}-y_{21}}}{l^{y_{11}}} + \frac{x_{21}}{l^{y_{21}}} = 0.
    \] So for $x_{21} l^{y_{11} - y_{21}} \in \Z$, we have $\displaystyle |a_1^{- x_{21} l^{y_{11} - y_{21}}} a_2| = \prod_{j=2}^n t_j^{\lambda_j}$, where $\lambda_j \in \Q$. Choose $c_1 := a_1$. So for all $i$, $\displaystyle |a_1^{- x_{i1} l^{y_{11} - y_{i1}}} a_i| = \prod_{j=2}^n t_j^{\lambda_{ij}}$, where $\lambda_{ij} \in \Q$.

    Continuing this process,  we get $c_2, \dots , c_n$,  $b_i \in k^*$  and $\mu_{ij} \in \Z$ such that 
$$ \mid \! a_i b_i^\ell \!  \mid =  \mid  \!  \prod_{j=1}^n c_j^{\mu_{ij}} \! \mid . $$
Since for any $u \in k^*$, $u$ is a unit in the valuation ring of $k$ if and only if $\mid \! u  \! \mid = 1$,
we have 
$$a_i =  u_i  b_i^\ell \prod_j c_j^{\mu_{ij}} $$
 for some units $u_i \in k^\circ$, the valuation ring of $k$, $b_i \in k^*$ and 
$\mu_{ij} \in \Z$. 

\end{proof}

\begin{theorem} \label{leq 2^n}
    Let $k$ be a complete ultrametric field with $\text{char}(\Tilde{k}) \neq 2$. Suppose that $\text{dim}_\Q (\sqrt{\mid k^* \mid}) = n$ is finite. Then $u(k) \leq 2^n u(\Tilde{k})$.
\end{theorem}
\begin{proof}
    Given a quadratic form $\displaystyle q = \langle a_1, \dots, a_s \rangle$ over $k$ of dimension $s > 2^n u(\Tilde{k})$, such that $\mid a_i \mid \in \sqrt{\mid k^* \mid}$ for all $i$. From the main decomposition theorem, we have elements $c_1, \dots, c_n$ in $k^*$ such that each $a_i = \displaystyle u_i \prod_{j=1}^n {c_j}^{\lambda_{ij}} b_i^2$, for $u_i$ unit in $(k^\circ)^*$, $b_i \in k^*$, and $\lambda_{ij} \in \{ 0, 1\}$.
    
    So there are $2^n$ possibilities for $\displaystyle \prod_{j=1}^n c_j^{\lambda_{ij}}$. Let's call them $\theta_1, \theta_2, \dots, \theta_{2^n}$. Thus each $a_i \equiv u_i \theta_j (\mod{(L^*)^2})$, for some $j$ with $1 \leq j \leq 2^n$.

    So after re-indexing, we have, $\displaystyle \langle a_1, \dots, a_s \rangle \cong \langle u_1, \dots, u_{s_1} \rangle \theta_1 \perp \langle u_{s_1 + 1}, \dots, u_{s_2} \rangle \theta_2 \perp \dots \perp \langle u_{s_{2^{n-1}+1}}, \dots, u_{s_{2^n}} \rangle \theta_{2^n}$, where all the $s_i$'s add up to $s$. Consequently, for any $\theta_i$ as above, there exists a diagonal quadratic form $\displaystyle \sigma_{\theta_i}$ with coefficients in $(k^\circ)^*$ such that $q$ is $k$-isometric to $\displaystyle \perp_{\sigma \in Q} \theta_i \cdot \sigma_{\theta_i}$. 

    Since $s > 2^n u(\Tilde{k})$, there exists $i$ such that $\displaystyle \text{dim}(\sigma_{\theta_i}) = \text{dim}(\langle u_{s_i + 1}, \dots, u_{s_{i+1}} \rangle) = s_{i+1} - s_i > u(\Tilde{k})$. Thus $\sigma_{\theta_i}$ is isotropic over $\Tilde{k}$. So $\sigma_{\theta_i}$ is isotropic over $k$. Thus $q$ is isotropic over $k$, making $u(k) \leq 2^n u(\Tilde{k})$.    
\end{proof}

 This gives us a refinement of the results on $u$-invariant from Mehmeti's paper \cite{Meh19}. 


\begin{theorem} \label{A split}
    Let $k$ be an ultrametric field such that $\text{dim}_\Q (\sqrt{\mid k^* \mid}) = n$ is finite. Let  $\Tilde{k}$ be the residue field of $k$ and $l$ such that $l \neq \text{char}(k)$. Let $\mu_l \subset k$. Let $A \in \ _l\text{Br}(k)$. Then there exists $a_1, a_2, \dots, a_n$ and $b_1, b_2, \dots, b_n$ in $k$; $A_0 \in \ _l\text{Br} (k^\circ)$ such that \[
   \displaystyle A = A_0 \otimes \prod_{i=1}^n (a_i, b_i)_l .
    \]
\end{theorem}
\begin{proof}
    Let $A \in \ _l\text{Br}(k)$. Merkurjev-Suslin's theorem implies that $A = \displaystyle \prod_{i=1}^m (\alpha_i, \beta_i)_l$ for some $m \in \mathbb{N}$. By Theorem \ref{main decomposition}, there exists $c_1, c_2, \dots, c_n$ in $k^*$ such that $\alpha_i = u_i \prod c_j^{r_{ij}}$ and $\beta_i = v_i \prod c_j^{s_{ij}}$, where $r_{ij}$ and $s_{ij}$ are integers and $u_i, v_i$ are units.

    So, \begin{align*} (\alpha_i, \beta_i) = (u_i \prod c_j^{r_{ij}}, v_i \prod c_j^{s_{ij}}) = (u_i, v_i) \otimes \sum (u_i, c_j^{s_{ij}}) \otimes \sum (c_j^{r_{ij}}, v_i) \otimes \sum (\prod c_j^{r_{ij}}, c_j^{s_{ij}}).
   \end{align*}
   Let $\prod c_j^{r_{ij}} = d_i$. So we get that
   \begin{align*}
      (\alpha_i, \beta_i) = (u_i, v_i) \otimes \sum (u_i, c_j^{s_{ij}}) \otimes \sum (c_j^{r_{ij}}, v_i) \otimes \sum (d_i, c_j^{s_{ij}})\\
      = (u_i, v_i) \otimes \sum (u_i^{s_{ij}}, c_j) \otimes \sum (v_i^{-r_{ij}}, c_j) \otimes \sum (d_i^{s_{ij}}, c_j),
   \end{align*} the last equality coming from properties of symbols, i.e., $(\alpha, \beta^j) = (\alpha^j, \beta)$ and $(\alpha,\beta) = - (\beta,\alpha) = (\beta^{-1}, \alpha)$. Now combining the terms, we get the following equality
   \begin{align*}
     (\alpha_i, \beta_i)  = (u_i, v_i) \otimes \displaystyle \sum_{j=1}^n (u_i^{s_{ij}} v_i^{-r_{ij}} d_i^{s_{ij}}, c_j).
   \end{align*} Let $x_j = u_i^{s_{ij}} v_i^{-r_{ij}} d_i^{s_{ij}}$ in $k^*$. Thus, we have $A = \displaystyle \prod_{i=1}^m \big( (u_i, v_i) \otimes \sum_{j=1}^n (x_j, c_j) \big)$.

   Let $\displaystyle A_0 = \otimes (u_i, v_i)_{k^\circ}$. Since each cyclic algebra $\displaystyle (u_i, v_i)_{k^\circ}$ is an Azumaya algebra over $k^\circ$, $A_0 \in \text{Br}(k^\circ)$ and we thus have $\displaystyle A = A_0 \otimes \prod_{i=1}^n (a_i, b_i)$.
 
\end{proof}

Recall the following definition. Let $K$ be any field. For a prime $l$, define the \textit{Brauer $l$-dimension} of $K$, denoted by Br$_l$dim($K$), to be the smallest integer $d \geq 0$ such that for every finite extension $L$ of $K$ and for every central simple algebra $A$ over $L$ of period a power of $l$, ind($A$) divides per($A$)$^d$.

\begin{theorem}\label{leq d+n}
    Let $k$ be a complete ultrametric field. Suppode $\text{dim}_\mathbb{Q} (\sqrt{|k^*|}) = n$ is finite. Let $l$ be a prime such that $l \neq \text{char}(\Tilde{k})$ and $\mu_l \subset k$. Then the $\text{Br}_l \text{dim} (k) \leq \text{Br}_l \text{dim} (\Tilde{k}) + n$.
\end{theorem}

\begin{proof} 
    Consider the $l$-torsion of Br$(k)$, $_l\text{Br}(k)$. Then for a given $A \in \ _l \text{Br} (k)$, by Theorem \ref{A split}, there exists $A_0 \in \ _l \text{Br}(k^\circ)$ and $c_1, c_2, \dots, c_n \in k^*$ such that $(A - A_0) \otimes\ k(\displaystyle \sqrt[l]{c_1}, \sqrt[l]{c_2}, \dots, \sqrt[l]{c_n})$ is split. Let $L = k(\displaystyle \sqrt[l]{c_1}, \sqrt[l]{c_2}, \dots, \sqrt[l]{c_n})$, so $(A - A_0) \otimes L$ is split.

    This implies that $\text{ind}(A)$ divides $\text{ind}(A_0) l^{n}$. 
    Since $k^\circ$ is complete, it follows from Proposition \ref{index of azumaya algebras} that $\displaystyle \text{ind}(A_0) \leq \text{ind}(A_0 \otimes_{k^\circ} \Tilde{k})$, hence divides $\text{Br}_l \text{dim}(\Tilde{k})$.  So, $\text{Br}_l \text{dim} (k) \leq \text{Br}_l \text{dim} (\Tilde{k}) + n$.

\end{proof}

\begin{cor}
    Let $k$ be a complete ultrametric field. Suppose $\text{dim}_\mathbb{Q} (\sqrt{|k^*|}) = n$ is finite. Let $l$ be a prime such that $l \neq \text{char}(\Tilde{k})$. Then the $\text{Br}_{l} \text{dim} (k) \leq \text{Br}_{l} \text{dim} (\Tilde{k}) + n$.
\end{cor}
\begin{proof} 
Let $A \in\ _l\text{Br}(k)$. Let $\rho_l$ be the $l$-th primitive root of unity. So, the degree of the extension $k(\rho_l)$ over $k$ is at most $l-1$, i.e., $l \nmid [k(\rho_l) : k]$. Thus, ind($A$) $=$ ind($A \otimes k(\rho_l)$) and per($A$) $=$ per($A \otimes k(\rho_l)$), by \cite{P82}, Propositions 13.4(vi) and 14.4b(v). Hence the corollary follows from Theorem \ref{leq d+n}.

\end{proof}

\section{Function field of curves on complete ultrametric fields}

Recall the discussion at the end of Section 2 regarding the rational rank of abelian groups. In our situation, consider a complete ultrametric valued field $k$. Let $L$ be a valued field extension of $k$. Let $|L^*|$, $|k^*|$ be the value groups of $L$ and $k$, respectively. We thus have $\text{dim}_\Q (\sqrt{|L^*|}) = \text{dim}_\Q (\sqrt{|k^*|}) + \text{dim}_\Q (\frac{|L^*|}{|k^*|} \otimes_\Z \Q)$. If $L$ is an algebraic extension of $k$, then the quotient group $|L^*|/|k^*|$ is a torsion group. Thus $\text{dim}_\Q (|L^*|/|k^*| \otimes_\Z \Q) = 0$, which implies that the two divisorial closures $\sqrt{|L^*|}$ and $\sqrt{|k^*|}$ have the same dimension over $\Q$.

\begin{theorem} \label{leq 2^n+1}
    Let $k$ be a complete ultrametric field with $\text{char}(\Tilde{k}) \neq 2$. Suppose $\text{dim}_\Q (\sqrt{\mid k^* \mid}) = n$ is finite. Let $C$ be a curve over $k$ and $F = k(C)$ the function field of the curve. Let $q$ be a quadratic form over $F$ with dimension $d$. If $d > 2^{n+1} u_s(\Tilde{k})$, then $q$ is isotropic. In particular, $u(F) \leq 2^{n+1} u_s(\Tilde{k})$.
\end{theorem}

\begin{proof}
Let $V(F)$ be the set of all non-trivial rank 1 valuations of $F$ which either extend the valuation of $k$ or are trivial when restricted to $k$. Let $F_v$ denote the completion of $F$ with respect to $v$. We will first show that for a quadratic form $q$ over $F$ with $\text{dim} (q) > 2^{n+1} u_s(\Tilde{k})$, $q$ is isotropic over $F_v$ for all $v \in V(F)$. Then we will apply Corollary 3.19 from \cite{Meh19} to say that $q$ is isotropic over $F$.

    Since we want to say that $q$ is isotropic over $F_v$ for all $v \in V(F)$, we have two cases. If $v \in V(F)$ is such that $v$ restricted to $k$ is the trivial valuation, then $F_v$ is a discrete valued field. Let $\widetilde{F_v}$ be the residue field of $F_v$. By Springer’s theorem on nondyadic complete discrete valuation fields (see \cite{Lam}, VI.1.10 and XI.6.2(7)), we have $u(F_v) = 2u(\widetilde{F_v})$. Since $\widetilde{F_v}$ is a finite extension of $k$, $\widetilde{F_v}$ is a complete ultrametric field with residue field $\widetilde{\widetilde{F_v}}$ a finite field extension of $\Tilde{k}$. Using the fact that $\widetilde{F_v}$ is an algebraic extension of $k$, then the quotient group $|\widetilde{F_v}^*|/|k^*|$ is a torsion group. Thus $\text{dim}_\Q (|\widetilde{F_v}^*|/|k^*| \otimes_\Z \Q) = 0$, which in turn implies that the two divisorial closures $\sqrt{|\widetilde{F_v}^*|}$ and $\sqrt{|k^*|}$ have the same dimension over $\Q$. It then follows from Theorem \ref{leq 2^n} that $u(\widetilde{F_v}) \leq 2^n u(\widetilde{\widetilde{F_v}}) \leq 2^n u_s (\Tilde{k})$. Thus, $u(F_v) \leq 2^{n+1} u_s(\Tilde{k})$.

         In the case that $v$ restricted to $k$ is the valuation on $k$, then $\widetilde{F_v}$ is an extension of $\Tilde{k}$ of transcendence degree $\leq 1$. Let $ \displaystyle s = \text{dim}_\Q \Big( \frac{\mid F_v^* \mid}{\mid k^* \mid} \otimes_\Z \Q \Big)$ and $t = \text{tr deg}_{\Tilde{k}} (\widetilde{F_v})$. By Abhyankar's inequality \cite{Abh}, we have $0 \leq s + t \leq 1$. 

         Suppose $t=0$. Then $\widetilde{F_v}$ is a finite extension of $\Tilde{k}$, and hence $u(\widetilde{F_v}) \leq u_s(\Tilde{k})$. Since $t=0$ and $0 \leq s+t \leq 1$, we have $s \leq 1$. When $s=0$, the discussion before Theorem 4.1 tells us that, $\text{dim}_\Q (\sqrt{|F_v^*|}) = \text{dim}_\Q (\sqrt{|k^*|}) + \text{dim}_\Q (\frac{|F_v^*|}{|k^*|} \otimes_\Z \Q) = \text{dim}_\Q (\sqrt{|k^*|}) + s$. Thus $\text{dim}_\Q (\sqrt{|F_v^*|}) = \text{dim}_\Q (\sqrt{|k^*|}) = n$. From Theorem \ref{leq 2^n}, we thus have $u(F_v) \leq 2^n u(\widetilde{F_v}) \leq 2^n u_s(\Tilde{k})$. When $s=1$, again the discussion before Theorem 4.1 implies that, $\text{dim}_\Q (\sqrt{|F_v^*|}) = \text{dim}_\Q (\sqrt{|k^*|}) + s$. Thus $\text{dim}_\Q (\sqrt{|F_v^*|}) = n + 1$. Applying Theorem \ref{leq 2^n} again, we have $u(F_v) \leq 2^{n+1} u(\widetilde{F_v}) \leq 2^{n+1} u_s(\Tilde{k})$. 
         
         For the case when $t=1$ and $s=0$, the extension $\widetilde{F_v}$ over $\Tilde{k}$ is finitely generated of transcendence degree $1$. Thus, $u(\widetilde{F_v}) \leq 2 u_s(\Tilde{k})$. Once again applying Theorem \ref{leq 2^n}, we get that $u(F_v) \leq 2^n u(\widetilde{F_v}) \leq 2^{n+1} u_s(\Tilde{k})$.
       
    Now applying \ref{LGP Meh}, we get that $q$ is isotropic over $F$. This proves the claim in the theorem that $u(F) \leq 2^{n+1} u_s (\Tilde{k})$.
\end{proof}

\begin{cor}
   Let $k$ be a complete ultrametric field with $\text{char}(\Tilde{k}) \neq 2$. Suppose $\text{dim}_\Q (\sqrt{\mid k^* \mid}) = n$ is finite. Then $u_s(k) \leq 2^n u_s(\Tilde{k})$.
\end{cor}
\begin{proof}
    Let $K/k$ be a finite field extension. Then $K$ is an ultrametric field and $\Tilde{K}/ \Tilde{k}$ is a finite extension. Let $q$ be a $K$-quadratic form of dimension $d > 2^n u_s(\Tilde{k})$. Since $\text{char}(\Tilde{k}) \neq 2$, we may assume $q$ to be diagonal. Further, $\text{dim}_\Q (\sqrt{|K^*|}) = \text{dim}_\Q (\sqrt{|k^*|}) = n$ is finite. This equality follows from the discussion before Theorem 4.1, and noting that since $K$ is a finite extension of $k$, the quotient group $|K^*|/|k^*|$ is torsion. Then applying Theorem \ref{leq 2^n}, we have $u(K) \leq 2^n u(\Tilde{K}) \leq 2^n u_s (\Tilde{k})$.

    Let $C$ be a curve over $k$ and $F = k(C)$ the function field of $C$. Then it follows from Theorem \ref{leq 2^n+1} that $u(F) \leq 2^{n+1} u_s(\Tilde{k})$. Then by the definition of strict $u$-invariant, we have that $u_s(k) \leq 2^n u_s(\Tilde{k})$.
\end{proof}

\begin{theorem} \label{leq d+n+1}
    Let $k$ be a complete ultramteric field. Suppose $\text{dim}_\mathbb{Q} (\sqrt{|k^*|}) = n$ is finite. Let $l$ be a prime such that $l \neq \text{char}(\Tilde{k})$. Let $C$ be a curve over $k$ and $F = k(C)$ the function field of the curve. Suppose there exist an integer $d$ such that $\text{Br}_l \text{dim} (\Tilde{k}) \leq d$ and $\text{Br}_l \text{dim} (\Tilde{k}(T)) \leq d+1$. Then the $\text{Br}_l \text{dim} (F) \leq d+1+n$.
\end{theorem}

\begin{proof}
    Let $V(F)$ be the set of all non-trivial rank 1 valuations of $F$ which either extend the valuation of $k$ or are trivial when restricted to $k$. Let $F_v$ denote the completion of $F$ with respect to $v$. 

    If $v \in V(F)$ is such that $v$ restricted to $k$ is the trivial valuation, then $F_v$ is a discrete valued field. Let $\widetilde{F_v}$ be the residue field of $F_v$. Since $\widetilde{F_v}$ is a finite extension of $k$, it is a complete ultrametric field with residue field $\widetilde{\widetilde{F_v}}$ a finite extension of $\Tilde{k}$. By the definition of $\text{Br}_l \text{dim}$, we have that $\text{Br}_l \text{dim}(\widetilde{\widetilde{F_v}}) \leq \text{Br}_l \text{dim}(\Tilde{k})$. Note that $F_v$ is a complete discretely valued field, thus $\text{Br}_l \text{dim} (F_v) \leq \text{Br}_l \text{dim}(\widetilde{F_v}) + 1$, which follows from \cite[Theorem 5.5]{HHK09}. Now utilising the fact that $\widetilde{F_v}$ is an algebraic extension of $k$, the quotient group $|\widetilde{F_v}^*|/|k^*|$ is a torsion group. Thus, $\text{dim}_\Q (|\widetilde{F_v}^*|/|k^*| \otimes_\Z \Q) = 0$, which in turn implies that the two divisorial closures $\sqrt{|\widetilde{F_v}^*|}$ and $\sqrt{|k^*|}$ have the same dimension over $\Q$. Applying Theorem \ref{leq d+n}
 to $\widetilde{F_v}$, we have that $\text{Br}_l \text{dim} (\widetilde{F_v}) \leq \text{Br}_l \text{dim} (\widetilde{\widetilde{F_v}}) + n$. Putting it all together gives $\text{Br}_l \text{dim}(F_v) \leq \text{Br}_l \text{dim}(\widetilde{F_v}) + 1 \leq \text{Br}_l \text{dim} (\widetilde{\widetilde{F_v}}) + n + 1 \leq \text{Br}_l \text{dim}(\Tilde{k}) + n + 1 \leq d+n+1$.

    In the case that $v$ restricted to $k$ is the valuation on $k$, then $\widetilde{F_v}$ is an extension of $\Tilde{k}$ of transcendence degree $\leq 1$. Let $ \displaystyle s = \text{dim}_\Q \Big( \frac{\mid F_v^* \mid}{\mid k^* \mid} \otimes_\Z \Q \Big)$ and $t = \text{tr deg}_{\Tilde{k}} (\widetilde{F_v})$. By Abhyankar's inequality \cite{Abh}, we have $0 \leq s + t \leq 1$. 

    Suppose $t=0$. Then $\widetilde{F_v}$ is a finite extension if $\Tilde{k}$. Thus, $\text{Br}_l \text{dim}(\widetilde{F_v}) \leq \text{Br}_l \text{dim}(\Tilde{k})$. Since $t=0$ and $0 \leq s+t \leq 1$, we have $s \leq 1$. When $s=0$, the discussion at the beginning of Section 4 implies that $\text{dim}_\Q (\sqrt{|F_v^*|}) = \text{dim}_\Q (\sqrt{|k^*|}) + \text{dim}_\Q (\frac{|F_v^*|}{|k^*|} \otimes_\Z \Q) = \text{dim}_\Q (\sqrt{|k^*|}) + s$. Thus $\text{dim}_\Q (\sqrt{|F_v^*|}) = \text{dim}_\Q (\sqrt{|k^*|}) = n$. Applying Theorem \ref{leq d+n} now gives that, $\text{Br}_l \text{dim}(F_v) \leq \text{Br}_l \text{dim}(\widetilde{F_v}) + n \leq \text{Br}_l \text{dim}(\Tilde{k}) + n \leq d+n$. 
    When $s=1$, once again the discussion at the beginning of Section 4 tells us that, $\text{dim}_\Q (\sqrt{|F_v^*|}) = \text{dim}_\Q (\sqrt{|k^*|}) + s$. Thus $\text{dim}_\Q (\sqrt{|F_v^*|}) = n + 1$. By Theorem \ref{leq d+n}, $\text{Br}_l \text{dim}(F_v) \leq \text{Br}_l \text{dim}(\widetilde{F_v}) + \text{dim}_\Q (\sqrt{|F_v^*|}) \leq \text{Br}_l \text{dim}(\Tilde{k}) + n+1 \leq d+n+1$.

         For the case when $t=1$ and $s=0$, the extension $\widetilde{F_v}$ over $\Tilde{k}$ is finitely generated of transcendence degree $1$, \textit{i.e.,} $\widetilde{F_v}$ is a finite extension of $\Tilde{k}(T)$. 
         So, $\text{Br}_l \text{dim}(\widetilde{F_v}) \leq \text{Br}_l \text{dim} (\Tilde{k}(T))$. By Theorem \ref{leq d+n}, $\text{Br}_l \text{dim}(F_v) \leq \text{Br}_l \text{dim}(\widetilde{F_v}) + n \leq \text{Br}_l \text{dim}(\Tilde{k}(T)) + n \leq d+1+n$.

         Thus for all $v \in V(F)$, we have $\text{Br}_l \text{dim}(F_v) \leq d+n+1$. From Theorem \ref{local-global index}, we see that $\text{Br}_l \text{dim}(F) \leq \text{Br}_l \text{dim}(F_v)$, proving the claim of the theorem.

\end{proof}

\providecommand{\bysame}{\leavevmode\hbox to3em{\hrulefill}\thinspace}

\end{document}